\documentclass[11pt]{article}
\usepackage{graphicx} % Required for inserting images
\usepackage{hyperref}
\usepackage{amssymb}
\usepackage{caption}
\usepackage{subcaption}
\usepackage{amsmath}
\usepackage{amsfonts}
\usepackage{amsthm}
\usepackage{stmaryrd}
\usepackage{algorithm, algorithmic,refcount}
\usepackage{tkz-euclide}

\usepackage{epstopdf}
\usepackage{mathdots}
\usepackage{appendix}
\usepackage[shortlabels]{enumitem}
\usetikzlibrary{patterns}
\usepackage{xcolor}
\usepackage{natbib}
\setcitestyle{square}
\setcitestyle{numbers}
\setcitestyle{comma}
\usepackage{graphicx}

\newtheorem{theorem}{Theorem}
\newtheorem{example}[theorem]{Example}
\newtheorem{corollary}[theorem]{Corollary}
\newtheorem{lemma}[theorem]{Lemma}
\newtheorem{remark}[theorem]{Remark}
\newtheorem{definition}[theorem]{Definition}
\newtheorem{assumption}[theorem]{Assumption}

\title{A sublinear-time randomized algorithm for column and row subset selection based on strong rank-revealing QR factorizations}
\author{Alice Cortinovis\footnote{Department of Mathematics, Stanford University, CA, USA. Email: \texttt{alicecor@stanford.edu}} \and Lexing Ying\footnote{Department of Mathematics and ICME, Stanford University, CA, USA. Email: \texttt{lexing@stanford.edu}}}

\date{}

\begin{document}
	
	\maketitle
	
	\begin{abstract}
		In this work, we analyze a sublinear-time algorithm for selecting a few rows and columns of a matrix for low-rank approximation purposes. The algorithm is based on an initial uniformly random selection of rows and columns, followed by a refinement of this choice using a strong rank-revealing QR factorization. We prove bounds on the error of the corresponding low-rank approximation (more precisely, the CUR approximation error) when the matrix is a perturbation of a low-rank matrix that can be factorized into the product of matrices with suitable incoherence and/or sparsity assumptions.
	\end{abstract}
	
	\section{Introduction}
	Many large-scale matrices arising in applications have a low numerical rank (see, e.g.,~\cite{Udell2019}), and while the truncated singular value decomposition gives a way to construct the \emph{best} low-rank approximation with respect to all unitarily invariant norms (by the Eckart-Young-Mirsky theorem~\cite{Eckart1936, Mirsky1960}), this is often too expensive to compute. For this reason, different types of low-rank approximation strategies have been analyzed in the literature, for example, approximations constructed from some rows and columns of the matrix. This work is concerned with the analysis of a randomized algorithm that selects suitable rows and columns.
	
	Let $A \in \mathbb{R}^{n \times m}$ be the matrix we want to approximate. 
	Let us denote by $I \in \{1, \ldots, n\}^\ell$ and $J \in \{1, \ldots, m\}^\ell$  ordered index sets that correspond to rows and columns of $A$, respectively, for some $\ell \ll \min \{m,n\}$, and let us denote by $A(I,:) \in \mathbb{R}^{\ell \times m}$ and $A(:,J) \in \mathbb{R}^{n \times \ell}$ the submatrices of $A$ corresponding to the rows indexed by $I$ and the columns indexed by $J$, respectively. An approximation of $A$ using these rows and columns has the form
	\begin{equation}\label{eq:lowrank}
		A \approx A(:,J) M A(I,:),
	\end{equation}
	for some matrix $M \in \mathbb{R}^{\ell \times \ell}$. The choice of $M$ that minimizes the low-rank approximation error $\|A - A(:,J) M A(I,:)\|_F$ in the Frobenius norm (given a choice of $I$ and $J$) is the orthogonal projection $M = A(:,J)^{\dagger} A A(I,:)^{\dagger}$ (see, e.g.,~\cite{Stewart1999}), where $\dagger$ denotes the Moore-Penrose pseudoinverse of a matrix. The resulting approximation is usually called a ``CUR approximation''~\cite{Drineas2008,Saibaba2016,Sorensen2016}. In this case, the whole matrix $A$ is needed in order to compute $M$. A less expensive choice for $M$ is $M = A(I,J)^\dagger$, or a regularized version of this quantity,  
	where $A(I,J)$ denotes the $\ell \times \ell$ matrix formed by the intersection of the rows indexed by $I$ and the columns indexed by $J$. This is called ``cross approximation''~\cite{Bebendorf2000} or  ``double-sided interpolatory decomposition''~\cite{Dong2023}, because the resulting low-rank approximation  $A(:, J)A(I, J)^\dagger A(I, :)$ coincides with $A$ in the selected rows and columns.
	
	\subsubsection*{Row and column selection strategies} 
	The quality of the low-rank approximation, that is, the norm of the error matrix  $$A - A(:,J)MA(I,:),$$ depends on the choice of rows and columns. The CUR approximation error can be easily related to the column subset selection error and the row subset selection error:
	\begin{equation*}
		\|A - A(:,J)A(:,J)^\dagger A\|, \qquad \|A - A A(I,:)^\dagger A(I,:)\|.
	\end{equation*}
	If we consider the spectral norm, which we denote by $\| \cdot \|$, it is possible~\cite{Gu1996} to choose an index set $J$ of cardinality $\ell$ such that 
	\begin{equation}\label{eq:bestspectral}
		\|A - A(:,J)A(:,J)^\dagger A\| \le \sqrt{(\ell+1)(n-\ell)} \sigma_{\ell+1}(A),
	\end{equation}
	where $\sigma_{\ell+1}(A)$ denotes the $(\ell+1)$-th singular value of $A$ and is the error of the best rank-$\ell$ approximation in the spectral norm.
	
	In practice, the strategy for choosing rows and columns depends on the properties and the size of the matrix.  In the literature, several deterministic and randomized strategies for selecting rows and columns for CUR approximation have been developed.  If it is feasible to look at all the entries of the matrix, one option is the Adaptive Cross Approximation algorithm~\cite{Bebendorf2000}, a deterministic greedy algorithm that gives good results in practice on a wide range of matrices but can fail for a few others (see, e.g.,~\cite{CKM2020}) and takes time $O(nm\ell)$ for a rank-$\ell$ approximation. We refer to the reader to, e.g.,~\cite{Cortinovis2020, Drineas2008,Saibaba2016,Sorensen2016}, for discussions on deterministic techniques for column subset selection and CUR approximation.  
	
	In this paper, we are concerned with the case in which the matrix is so large that it is not possible (or it is impractical) to store it in memory or compute all its entries: randomized algorithms for sampling columns and rows come to the rescue. Among the strategies proposed in the literature, there are uniform sampling~\cite{Chiu2013}, sampling according to the norm of columns and rows of $A$~\cite{Frieze2004}, leverage score sampling~\cite{Cohen2017}, volume sampling~\cite{Cortinovis2020,Deshpande2006}, pivoting on random sketches~\cite{Dong2023,Voronin2017}, discrete empirical interpolation method~\cite{Sorensen2016}. 
	
	\subsubsection*{Uniform sampling and rank-revealing factorizations}
	The simplest method, uniform sampling, gives good low-rank approximations in many cases of interest. More specifically, Chiu and Demanet \cite{Chiu2013} proved that if $A \approx XZY^T$ for some rank-$k$ matrices $X \in \mathbb{R}^{n\times k}, Y \in \mathbb{R}^{m \times k}, Z \in \mathbb{R}^{k \times k}$, where $X$ and $Y$ have orthonormal columns if $X$ and $Y$ are \emph{incoherent} (see Definition~\ref{def:incoherent}) then selecting $\ell$ rows and columns of $A$ uniformly at random gives a good low-rank approximation of $A$, if $\ell$ is chosen to be mildly larger than the target rank $k$. 
	
	When the matrix $A$ does not satisfy these incoherence assumptions, uniform column and row sampling can still be an excellent starting point for constructing a low-rank approximation of $A$ in sublinear time. In~\cite[Section 3.2]{Chiu2013}, the case in which only $X$ is incoherent is discussed: in this situation, a uniformly random row selection is followed by a strong rank-revealing QR (sRRQR) factorization of these rows to identify suitable columns. The resulting low-rank approximation is provably close to optimal. 
	
	A similar, more heuristic, approach was suggested in~\cite[Algorithm 2.2]{Li2015}: starting from index sets $I$ and $J$ of rows and columns sampled uniformly at random, one alternately considers the rows $A(I,:)$ and the columns $A(:,J)$ and applies a rank-revealing algorithm (QR with column pivoting) to get a new -- hopefully better -- set of columns and rows, respectively.  While this latter algorithm has been seen to perform very well in practice~\cite{Li2015}, the available theoretical results are not able to explain its behavior beyond the assumptions of~\cite{Chiu2013}.
	
	\subsubsection*{Contributions} In this work, we consider and analyze a variation of the column and row selection strategies described above. Namely, given a matrix $A$, we first select some rows uniformly at random, then perform an sRRQR factorization on these rows to get columns, and add some more columns selected uniformly at random to get the final column index set. An analogous strategy is used to select a row index set. The procedure is summarized in Algorithm~\ref{alg:variation}. We prove error bounds on the low-rank approximation given by Algorithm~\ref{alg:variation} under the assumption that $A$ is a perturbation of a low-rank matrix $XZY^T$, for some $X \in \mathbb{R}^{n \times k}$, $Z \in \mathbb{R}^{k \times k}$, $Y \in \mathbb{R}^{m \times k}$, where some of the columns of $X$ and $Y$ are sparse, and the other ones are incoherent. The precise statement is in Corollary~\ref{cor:variation}. In addition, we show that if there is no perturbation, that is, the matrix $A$ has rank exactly $k$, then Algorithm~\ref{alg:variation} exactly recovers $A$ with high probability.
	
	We also consider a second Algorithm (Algorithm~\ref{alg:basic}), inspired by~\cite[Algorithm 2.2]{Li2015}, that refines the choice of rows and columns iteratively. Under the same assumptions on $A$, Theorem~\ref{thm:1step} gives an error bound for the low-rank approximation corresponding to one iteration of this method.
	
	Both Algorithm~\ref{alg:variation} and Algorithm~\ref{alg:basic} -- for the selection of row and column index sets -- run in sublinear time with respect to the size of the matrix $A$. The construction of the corresponding CUR approximation, however, requires access to the full matrix $A$ or to a black-box computation of matrix-vector products with $A$.
	
	\subsubsection*{Outline} 
	The paper is organized as follows. In Section~\ref{sec:rrqrintro}, we recall the notion of a strong rank-revealing QR factorization and describe the randomized sRRQR algorithm for column and row subset selection, together with some conceptual examples. In Section~\ref{sec:preliminary}, we report some existing results on incoherence and projection onto column spaces that will be useful for our analysis. The error analysis for Algorithm~\ref{alg:variation} is presented in Sections~\ref{sec:exact} and~\ref{sec:approx}. In Section~\ref{sec:iterative}, we discuss the iterative version of the randomized sRRQR algorithm (Algorithm~\ref{alg:basic}) and show some numerical examples and an error bound. Section~\ref{sec:conclusion} concludes our paper with a discussion of the strengths and limitations of the obtained results.
	
	\section{A randomized sRRQR algorithm for column/row selection}\label{sec:rrqrintro}
	
	We start this section by recalling the definition of an sRRQR factorization and some consequences; we explain the randomized sRRQR algorithm and show its behavior on some simple example matrices.
	
	\subsection{Strong rank-revealing QR factorization}

	\begin{definition}[Strong rank-revealing QR factorization~\cite{Gu1996}]
		A rank-$k$ strong rank-revealing QR (sRRQR) factorization of $A \in \mathbb{R}^{n \times m}$ is given by
		\begin{equation*}
			AP = Q \begin{bmatrix} A_k & B_k \\ 0 & C_k \end{bmatrix}, 
		\end{equation*}
		where $P \in \mathbb{R}^{m \times m}$ is a permutation matrix, $Q \in \mathbb{R}^{n \times n}$ is an orthogonal matrix, $A_k \in \mathbb{R}^{k \times k}$, $B_k \in \mathbb{R}^{k \times (m-k)}$, $C_k \in \mathbb{R}^{(n-k)\times(m-k)}$, for which the following conditions hold for some $\eta \ge 1$:
		\begin{enumerate}[(a)]
			\item $\sigma_i(A_k) \ge \frac{\sigma_i(A)}{\sqrt{1+\eta k(n-k)}}$ for $1 \le i \le k$;
			\item $\sigma_j(C_k) \le \sigma_{k+j}(A)\sqrt{1+\eta k(n-k)}$ for $1 \le j \le n-k$;
			\item $\|A_k^{-1} B_k\|_{\max} \le 1$.
		\end{enumerate}
	\end{definition}
	
	Here, $\| \cdot \|_{\max}$ denotes the Chebyshev norm of a matrix, that is, the maximum absolute value of its entries.
	Intuitively, if the singular values $\sigma_{k+1}(A), \ldots, \sigma_m(A)$ of $A$ are small, this means that the columns of $A$ corresponding to the first $k$ columns of $P$ well describe the range of \emph{all} the columns of $A$. More precisely, let $J \in \{1, \ldots, n\}^k$ be the index set corresponding to the first $k$ columns of $P$. Then
	\begin{equation*}
		\|A - A(:,J)A(:,J)^\dagger A\| = \|C_k\| \le \sigma_{k+1}(A) \sqrt{1 + \eta k(n-k)}.
	\end{equation*}
	
	The paper~\cite{Gu1996} contains an algorithm for obtaining a sRRQR factorization of rank $k$, which runs in time $\mathcal O\left (mnk \log_{\eta}(\min\{n,m\})\right )$ for $\eta > 1$; see the discussion in~\cite[Section 4.4]{Gu1996}. In our error analysis in Sections~\ref{sec:erroranalysis} and~\ref{sec:iterative}, for ease of notation, we assume $\eta = 1$, but all the results can be easily modified to account for a small fixed constant $\eta$ (e.g., $\eta=2$).
	
	We choose to work with the sRRQR factorization because of its strong theoretical guarantees. Indeed, while pivoted LU or pivoted QR factorizations with column pivoting are slightly less computationally expensive and are often good strategies for finding the numerical rank of a matrix, there are pathological examples for which these algorithms fail to select a suitable set of columns. 
	
	\subsection{The algorithm}\label{sec:algorithm}
	
	We consider the following procedure for selecting rows and columns of a matrix $A \in \mathbb{R}^{n \times m}$ in sublinear time. We start from $\ell_0$ randomly selected rows of $A$, corresponding to an index set $I_0$. We perform the sRRQR algorithm on the submatrix $A(I_0,:) \in \mathbb{R}^{\ell_0 \times m}$ with parameter $\ell_a$, and select a column index set $J$ made of the corresponding columns selected by sRRQR and another $\ell_b$ columns selected uniformly at random. Now, to select the rows $I$, the same procedure is performed on the matrix $A^T$. The cost of the algorithm is $\mathcal O\left (mn\ell \log_{\eta}(\min\{n,m\})\right )$, where we set $\ell = \max\{\ell_0,\ell_a, \ell_b\}$; in particular, for small values of $\ell$ this is sublinear in the size of $A$. The number of chosen rows and columns should be slightly larger than the target rank for the low-rank approximation for this algorithm to make sense.
	
	\begin{algorithm}
		\begin{algorithmic}[1]
			\REQUIRE{Matrix $A$, number of indices $\ell_0, \ell_a, \ell_b$}
			\ENSURE{Row index set $I$ and column index set $J$ of cardinality $\ell_a + \ell_b$}
			\STATE{Select $\ell_0$ rows of $A$ uniformly at random (index set $I_0$)}
			\STATE{Select $\ell_a$ columns of $A(I,:)$ by sRRQR algorithm (index set $J_a$)}
			\STATE{Select another $\ell_b$ columns of $A$ uniformly at random (index set $J_b$)}
			\STATE{Select $\ell_0$ rows of $A^T$ uniformly at random (index set $J_0$)}
			\STATE{Select $\ell_a$ columns of $A^T(J_0,:)$ by sRRQR algorithm (index set $I_a$)}
			\STATE{Select another $\ell_b$ columns of $A^T$uniformly at random (index set $I_b$)}
			\STATE{Return the column index set $J = (J_a, J_b)$ and the row index set $I = (I_a, I_b)$}
			
		\end{algorithmic}
		\caption{Randomized sRRQR for row and column selection}
		\label{alg:variation}
	\end{algorithm}
	
	\begin{remark}
		Instead of choosing $\ell_a$ a priori, a more practical implementation would run sRRQR with a fixed accuracy and eliminate the ``less interesting'' rows/columns.
	\end{remark}
	
	\begin{remark}
		While sampling row and column index sets as in Algorithm~\ref{alg:variation} takes sublinear time in the matrix dimension, the process of computing the low-rank approximation 
		\begin{equation}\label{eq:proj}
			A(:,J)A(:,J)^\dagger A A(I,:)^\dagger A(I,:)
		\end{equation}
		exactly is fast only if the matrix $A$ admits a fast matrix-vector multiplication routine. If this is not the case, the time $\mathcal O(mn\ell)$ needed to compute~\eqref{eq:proj} may be too much in practice. A possible heuristic is then to sample some larger index sets $\mathcal I$ and $\mathcal J$ and take the middle matrix, which should be $A(:,J)^\dagger A A(I,:)^\dagger$, to be
		\begin{equation*}
			A(\mathcal I, J)^\dagger A(\mathcal I, \mathcal J) A(I, \mathcal J)^\dagger
		\end{equation*}
		instead, which is faster to compute; see Step 6 in Algorithm 2.2 from~\cite{Li2015}.
	\end{remark}

	\subsection{When is there hope for Algorithm~\ref{alg:variation} to work?}\label{sec:hope}
	
	We begin the discussion by trying to understand in which cases it is reasonable to hope that Algorithm~\ref{alg:variation} gives good rows and columns of $A$. We present some examples to illustrate its behavior. 
	
	When the matrix $A$ can be decomposed as the product of $XZY^T + E$ for some matrices $X \in \mathbb{R}^{n \times k}$ and $Y \in \mathbb{R}^{m \times k}$ with orthonormal columns and some $Z \in \mathbb{R}^{k \times k}$ and $E \in \mathbb{R}^{n \times m}$ with $\|E\|$ small, then it was proven in~\cite{Chiu2013} that if $X$ is incoherent, the first two lines of Algorithm~\ref{alg:variation} give provably good columns. Algorithm~\ref{alg:variation} works on a larger class of matrices, as we will show in Theorem~\ref{thm:exact}, Theorem~\ref{thm:approx}, and Corollary~\ref{cor:variation} below.

	On the other hand, there are cases in which Algorithm~\ref{alg:variation} is not able to select a good row and column subset, for instance, when the matrix $A$ is approximately a block matrix, and some blocks have a small size. The following is a simple conceptual example.
	
	\begin{example}\label{ex:notwork}
		Consider a matrix $A \in \mathbb{R}^{n \times n}$ that has entries $a_{11} = 1$, $a_{ij} = 1$ for $2 \le i,j \le n$, and $0$ otherwise. The matrix has rank $2$, but it is nearly impossible for Algorithm~\ref{alg:variation} to end up including the first row and column in the selected index sets $I$ and $J$. Indeed, the initial selection of $\ell_0$ rows will likely not include the first row, and since the first column will be zero in the selected submatrix, the first column will not be included in the column index set $J_a$. An analogous argument holds for the row index set $I$. 
	\end{example}
	
	The following example highlights the advantages of Algorithm~\ref{alg:variation} over a uniformly random selection of rows and columns and motivates the necessity for adding some new uniformly random indices in lines 3 and 6 of Algorithm~\ref{alg:variation}.
	
	\begin{example}\label{ex:rank2}
		Consider the matrix $A \in \mathbb{R}^{n \times n}$ that has entries $a_{1j} = a_{j1} = 1$ for $1 \le j \le n$ and zeros elsewhere. The row set $I_0$, chosen uniformly at random, will most likely not include the first row. However, when looking at the matrix $A(I_0,:)$, the sRRQR algorithm will select a set $J_a$ containing the first column, plus some other $\ell_a-1$ columns sampled uniformly at random. Now, the set $J$ will contain the first column and at least another column; therefore, it is enough to represent the range of $A$. An analogous argument holds for the row index set $I$.
	\end{example}
	
	Note that the matrix from Example~\ref{ex:rank2} has the following decomposition:
	\begin{equation*}
		A = \begin{bmatrix} \frac{1}{\sqrt n} & 1 \\ \frac{1}{\sqrt n} & 0 \\ \frac{1}{\sqrt n} & 0 \\ \vdots & \vdots \\ \frac{1}{\sqrt n} & 0  \end{bmatrix} \begin{bmatrix} \sqrt{n} & 0 \\ 0 & \sqrt{n-1} \end{bmatrix} \begin{bmatrix} 1 & 0 & 0 & \cdots & 0 \\ 0 & \frac{1}{\sqrt{n}} & \frac{1}{\sqrt{n}} & \cdots & \frac{1}{\sqrt{n}} \end{bmatrix} = X Z Y^T,
	\end{equation*}
	where, for each $j=1,2$, one between the $j$-th column of $X$ and $Y$ is sparse and the other is incoherent. This example suggests that when a matrix has a rank-$k$ decomposition $A = XZY^T$ up to an additive error $E$, there is hope for Algorithm~\ref{alg:variation} to work if, for each $i=1,\ldots,k$, one between the $i$-th columns of $X$ and of $Y$ is sparse, and the other is incoherent. 
	
	Motivated by the examples shown in this section, it makes sense to think that Algorithm~\ref{alg:variation} could return good row and column sets with high probability if $A$ admits a rank-$k$ approximation $XZY^T$, for some $X \in \mathbb{R}^{n \times k}$ and $Y \in \mathbb{R}^{m \times k}$, where the pairs of vectors of $X$ and $Y$ are either both incoherent or one sparse and one incoherent. Such a decomposition does not need to coincide with a singular value decomposition (or its truncation) of the matrix $A$. The assumptions are made precise in Section~\ref{sec:erroranalysis} below.
	
	\section{Error analysis of Algorithm~\ref{alg:variation}}\label{sec:erroranalysis}
	
	\subsection{Incoherence and useful preliminary results}\label{sec:preliminary}
	
	\begin{definition}\label{def:incoherent}[{\cite[Definition 1.1]{Chiu2013}}]
		Given a matrix $X \in \mathbb{R}^{n \times k}$ with orthonormal columns, the coherence of $X$ is $\mu := n \|X\|_{\max}^2.$ We say that $X$ is $\mu$-coherent.
	\end{definition}
	
	The coherence of an $n \times k$ matrix is always in the interval $[1, n]$, being $n$ when $X$ is a subset of columns of the identity matrix and being close to $1$ when the magnitude of all the entries of $X$ is more or less the same. We informally say that $X$ is incoherent if its coherence is small.
	
	In our analysis, we will use the following result, which states that randomly selecting rows from an incoherent matrix gives us a well-conditioned submatrix with high probability.
	\begin{theorem}[{\cite[Lemma 3.4]{Tropp2011}}]\label{thm:incoherent}
		Let $X \in \mathbb{R}^{n \times k}$ be a matrix with orthonormal columns, which is $\mu$-coherent. Let $\ell \ge \alpha k \mu$, and let $I \in \{1, \ldots, n\}^\ell$ be a subset of row indices chosen uniformly at random. Then
		\begin{equation*}
			\mathbb{P} \left ( \|X(I,:)^\dagger\| \ge \sqrt{\frac{n}{(1-\delta)\ell}}\right ) \le k \left ( \frac{e^{-\delta}}{(1-\delta)^{1-\delta}}\right )^\alpha \qquad \text{for any }\delta \in [0, 1)
		\end{equation*}
		and
		\begin{equation*}
			\mathbb{P} \left ( \|X(I,:)\| \ge \sqrt{\frac{(1+\delta')\ell}{n}}\right ) \le k \left ( \frac{e^{\delta'}}{(1+\delta')^{1+\delta'}}\right )^\alpha \qquad \text{for any }\delta' \ge 0.
		\end{equation*}
	\end{theorem}
	
	\begin{corollary}\label{cor:fullrank}
		Let $B \in \mathbb{R}^{n \times k}$ be a matrix, and let $X$ be an orthonormal basis of the column space of $B$. Let us assume that $X$ is $\mu$-coherent. Let $\ell \ge \alpha k \mu$, and let $I \in \{1, \ldots, n\}^\ell$ be a subset of row indices of $X$ chosen uniformly at random. Then the submatrix $B(I,:)$ is invertible with probability at least $1-ke^{-\alpha}$.
	\end{corollary}
	
	\begin{proof}
		This follows from the first part of Theorem~\ref{thm:incoherent} by taking the limit $\delta \to 1$.
	\end{proof}

	We will also use the following bound for the approximation of a matrix $A$ by the projection on a subset of its columns.
	\begin{theorem}[{\cite[Corollary 2.4]{Chiu2013}}]\label{thm:chiu}
		Let $A \in \mathbb{C}^{n \times m}$ and let $U = \begin{bmatrix} U_1 & U_2 \end{bmatrix}$  be a unitary matrix such that $U_1$ has $k$ columns. Let $J$ be an index set of $\ell \ge k$ indices. Assume that $U_1(:,J)$ has full column rank. Then
		\begin{equation*}
			\| A - A(:,J)A(:,J)^\dagger A \| \le \| U_1(:,J)^\dagger\| \cdot \| A U_2 U_2(:,J)^T\| + \|A U_2 \|.
		\end{equation*}
	\end{theorem}
	
	The following linear algebra results will be helpful as well.
	
	\begin{theorem}[{\cite[Exercise 7.3.P16]{Horn2013}}]\label{thm:singval}
		For any matrices $F$ and $G$ (of compatible size) we have
		\begin{enumerate}[(a)]
			\item $\sigma_i(FG) \le \|F\| \sigma_i(G)$ for all $i$, and
			\item $\sigma_{i+j-1}(F+G) \le \sigma_i(F) + \sigma_j(G)$ for all $i$ and $j$.
		\end{enumerate}
	\end{theorem}
	
	\begin{lemma}\label{lemma:sigmamintriangular}
		Let $X = \begin{bmatrix} A & B \\ 0 & C \end{bmatrix}$ be a block-triangular matrix where $A$, $B$, and $C$ have full column rank. 
		Assume that $\sigma_{\min}(A) \ge \delta_A$, $\|B\|_2 \le b$, and $\sigma_{\min}(C) \ge \delta_C$. Then
		\begin{equation*}
			\sigma_{\min}(X) \ge \frac{\delta_A \delta_C }{b} \left ( 1 + \frac{\delta_A^2 + \delta_C^2}{b^2}\right )^{-1/2}.
		\end{equation*}
	\end{lemma}
	\begin{proof}
		We have that $\sigma_{\min}(X)^{-1} = \sqrt{\|(X^T X)^{-1}\|_2}$, and $(X^T X)^{-1} = \begin{bmatrix} X_{11} & X_{12} \\ X_{21} & X_{22} \end{bmatrix}$, where
		\begin{align*}
			X_{11} & = (A^TA)^{-1} + (A^TA)^{-1}A^T B(B^TB+C^TC - B^TA (A^TA)^{-1}A^T B)^{-1} B^T A (A^TA)^{-1},\\
			X_{12} & = -(A^TA)^{-1}A^T B (B^TB+C^TC - B^TA (A^TA)^{-1}A^T B)^{-1},\\
			X_{21} & = - (B^TB+C^TC - B^TA (A^TA)^{-1}A^T B)^{-1} B^T A (A^TA)^{-1},\\
			X_{22} &= (B^TB+C^TC - B^TA (A^TA)^{-1}A^T B)^{-1}.
		\end{align*}
		We have that $\|(A^TA)^{-1}\| \le \delta_A^{-2}$ and $\|(C^TC)^{-1}\|\le \delta_C^{-2}$. Moreover, since $(A^TA)^{-1}A^T$ is the pseudoinverse of $A$, we have $\|(A^TA)^{-1} A^T\| = \|A (A^TA)^{-1}\| \le \sigma_{\min}(A)^{-1} \le \delta_A^{-1}$. Now note that
		\begin{equation*}
			B^TB + C^TC - B^T A(A^TA)^{-1}A^TB = C^TC + B^T(I - A(A^TA)^{-1}A^T)B \succeq CC^T,
		\end{equation*}
		since $B(I - A(A^TA)^{-1}A^T)B^T$ is symmetric positive semidefinite. Therefore, 
		\begin{equation*}
			\| (B^TB + C^TC - B^TA(A^TA)^{-1}A^TB)^{-1} \| \le \|(C^TC)^{-1}\| \le \delta_C^{-2}.
		\end{equation*}
		Thus, we can get upper bounds on the four blocks of $(X^TX)^{-1}$:
		\begin{align*}
			\|X_{11}\| & \le \delta_A^{-2} (1 + b^2\delta_C^{-2}),\\
			\|X_{12}\|, \|X_{21}\| & \le \delta_A^{-1}\delta_C^{-2} b,\\
			\|X_{22}\| & \le \delta_C^{-2}.
		\end{align*}
		Using the inequality $\|(X^TX)^{-1}\|^2 \le \|X_{11}\|^2 + \|X_{12}\|^2 + \|X_{21} \|^2 + \|X_{22}\|^2$ we obtain
		\begin{align*}
			\|(X^TX)^{-1}\|^2 & \le \delta_A^{-4}(1+b^2\delta_C^{-2})^2 + \delta_C^{-4} + 2 b^2 \delta_A^{-2} \delta_C^{-4} \\
			& \le \delta_A^{-4}(1+b^2\delta_C^{-2})^2 + \delta_C^{-4} + 2 \delta_A^{-2}(1 + b^2\delta_C^{-2})^2 \delta_C^{-2}\\
			& = (\delta_A^{-2}(1 + b^2\delta_C^{-2}) + \delta_C^{-2})^2 = (b^2 \delta_A^{-2} \delta_C^{-2} + \delta_A^{-2} + \delta_C^{-2})^2,
		\end{align*}
		from which we obtain
		\begin{equation*}
			\sigma_{\min}(X) \ge \frac{1}{\sqrt{b^2 \delta_A^{-2} \delta_C^{-2} + \delta_A^{-2} + \delta_C^{-2}}} = \frac{\delta_A \delta_C }{b} \left ( 1 + \frac{\delta_A^2 + \delta_C^2}{b^2}\right )^{-1/2},
		\end{equation*}
		which is the desired result.
	\end{proof}

	\subsection{Analysis for low-rank matrices}\label{sec:exact}
	
	In this section, we consider the case in which $A$ has a rank exactly $k$, and we want to find $\ell$ rows and columns that span \emph{exactly} the range of $A$ and $A^T$. Throughout this section, we will make the assumption that the implementation of the strong RRQR algorithm, when given a zero residual matrix, selects one of the remaining pivots uniformly at random.

	\begin{assumption}\label{ass:exact}
		We assume that $A = XZY^T$ for some $X\in \mathbb{R}^{n \times k}$, $Z \in \mathbb{R}^{k \times k}$, and $Y \in \mathbb{R}^{m \times k}$ that have rank exactly $k$. Let us assume that $Z = \mathrm{blkdiag}(Z_1, Z_2, Z_3)$, and that $X$ and $Y$ can be partitioned as 
		\begin{equation*}
			X = \begin{bmatrix} X_1 & X_2 & X_3 \end{bmatrix}, \qquad Y = \begin{bmatrix} Y_1 & Y_2 & Y_3 \end{bmatrix},
		\end{equation*}
		with $X_1\in \mathbb{R}^{n \times k_1}$, $Y_1\in \mathbb{R}^{m \times k_1}$, $X_2\in \mathbb{R}^{n \times k_2}$, $Y_2\in \mathbb{R}^{m \times k_2}$,  $X_3\in \mathbb{R}^{n \times k_3}$, $Y_3\in \mathbb{R}^{m \times k_3}$, $Z_1 \in \mathbb{R}^{k_1\times k_1}$, $Z_2 \in \mathbb{R}^{k_2\times k_2}$, $Z_3 \in \mathbb{R}^{k_3\times k_3}$, for some $k_1+k_2+k_3=k$. Let us further assume that $\begin{bmatrix} X_1 & X_2\end{bmatrix}$ and $\begin{bmatrix} Y_1 & Y_3\end{bmatrix}$ have orthonormal columns and are $\mu$-coherent and that each column of $X_3$ and $Y_2$ is sparse and has at most $\beta$ nonzero elements.
	\end{assumption}
	
	\begin{figure}[htb]
		\centering
		\includegraphics[scale=.22]{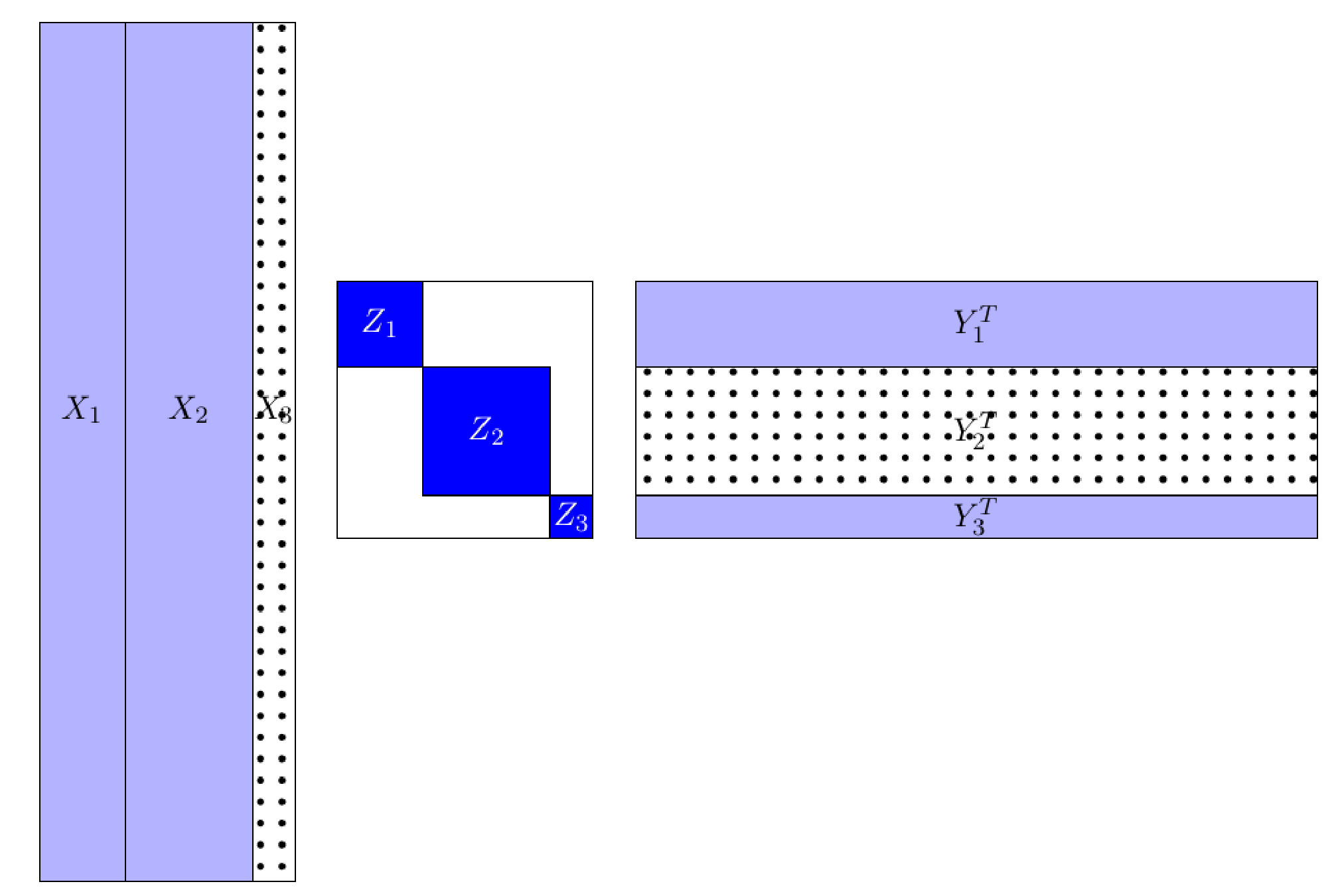}
%		\resizebox{180pt}{120pt}{
%			\begin{tikzpicture}[scale=.5]
%				\draw[fill=blue,opacity=.3] (0,0) rectangle (5,20);
%				\draw[fill=blue,opacity=.3] (14,8) rectangle (30,9);
%				\draw[fill=blue,opacity=.3] (14,12) rectangle (30,14);
%				\draw[pattern=dots] (5,0) rectangle (6,20);
%				\draw[pattern=dots] (14,9) rectangle (30,12);
%				\draw (0,0) rectangle (6, 20); % X
%				\draw (7,8) rectangle (13,14); % Z
%				\draw (14,8) rectangle (30, 14); % Y^T
%				\draw (2,0) -- (2,20);
%				\draw (5,0) -- (5,20);
%				\draw[fill=blue] (7,12) rectangle (9,14);
%				\draw[fill=blue] (9,9) rectangle (12,12);
%				\draw[fill=blue] (12,8) rectangle (13,9);
%				\draw (14, 9) -- (30,9);
%				\draw (14, 12) -- (30,12);
%				\node at (1,11) {$X_1$};
%				\node at (3.5,11) {$X_2$};
%				\node at (5.5,11) {$X_3$};
%				\node at (8,13) {\textcolor{white}{$Z_1$}};
%				\node at (10.5,10.5) {\textcolor{white}{$Z_2$}};
%				\node at (12.5,8.5) {\textcolor{white}{$Z_3$}};
%				\node at (22,13) {$Y_1^T$};
%				\node at (22,10.5) {$Y_2^T$};
%				\node at (22,8.5) {$Y_3^T$};
%		\end{tikzpicture}}
		\caption{Illustration of the partitioning of the factorization of the matrix $A$ from Assumption~\ref{ass:exact}. The dotted matrices are sparse, the light blue matrices are incoherent, and the dark blue matrices can be anything.}
		\label{fig:matrixA}
	\end{figure}

	\begin{theorem}\label{thm:exact}
		If Assumption~\ref{ass:exact} holds and we take
		\begin{equation*}
			\ell_0 = \alpha\mu k, \quad \ell_a = k, \quad \ell_b = \alpha\mu k,
		\end{equation*}
		then Algorithm~\ref{alg:variation} returns sets of rows $I$ and columns $J$ such that $$A = A(:,J)A(:,J)^\dagger A A(I,:)^\dagger A(I,:)$$ with probability at least
		\begin{equation*}
			1 - \left ( \frac{\beta \ell k}{n} + \frac{\beta \ell k}{m} + 2ke^{-\alpha} \right ),
		\end{equation*}
		where $\ell = \max\left \{ \ell_0, \ell_a,\ell_b\right \} = \alpha\mu k$.
	\end{theorem}
	
	\begin{proof}
		It is sufficient to show that, with high probability, lines 1--3 in Algorithm~\ref{alg:variation} select a column index set $A(:,J)$ that spans the range of $X$, and lines 4--6 select a row index set $A(I,:)$ that spans the range of $Y^T$. We will show the first fact only since the second one is completely analogous.

		In the first step of the algorithm, we select $\ell_0 \ge \alpha\mu(k_1+k_2)$ rows uniformly at random. By Corollary~\ref{cor:fullrank}, with probability at least $1- ke^{-\alpha} = 1 - p_1$ the submatrix $$\begin{bmatrix} X_1(I_0,:) &  X_2(I_0,:)\end{bmatrix}$$ has rank $k_1 + k_2$. By the sparsity assumption, with a probability of at least 
		\begin{equation}\label{eq:prob}
			1 - k_3\left ( 1 - \left ( 1 - \frac{\beta}{n}\right )^{\ell_0} \right ) = 1 - p_2 \ge 1 - \frac{\beta \ell_0 k_3}{n}
		\end{equation}
		the submatrix $X_3(I_0,:)$ is made of zeros. Indeed, for each column of $X_3$, by selecting $\ell_0$ columns uniformly at random (without replacement), we have a probability of at least $\left (1 - \frac{\beta}{n} \right )^{\ell_0}$ of sampling only zero entries, and a probability of at most $1 - \left ( 1- \frac{\beta}{n}\right )^{\ell_0}$ of selecting at least one nonzero entry; the bound~\eqref{eq:prob} follows from a union bound on the $k_3$ columns of $X_3$. 
		Therefore, using again a union bound, with probability at least $1-p_1-p_2$, the matrix $A(I_0,:)$ has rank exactly $k_1+k_2$. 
		
		The sRRQR algorithm gives an index set $J_a$ of cardinality $k_1 + k_2$, and from the rank-revealing property, we have that the top $(k_1+k_2)\times (k_1+k_2)$ matrix $A(I,J_a)$ has full rank (is invertible), therefore also 
		\begin{equation}\label{eq:small}
			\begin{bmatrix} Y_1(J_a,:) & Y_2(J_a,:)\end{bmatrix} \in \mathbb{R}^{(k_1+k_2) \times (k_1+k_2)}
		\end{equation}
		is invertible. The index set $J_b$ is chosen uniformly at random from $\{1, \ldots, n\}$. We can partition
		\begin{equation*}
			Y(J,:) = \begin{bmatrix} Y_1(J_a, :) & Y_2(J_a, :) & Y_3(J_a, :) \\
				Y_1(J_b, :) & Y_2(J_b, :) & Y_3(J_b, :) \end{bmatrix}.
		\end{equation*}
		The matrix $Y_2(J_a,:)$ has rank $k_2$ because it is a submatrix with $k_2$ columns of the invertible matrix~\eqref{eq:small}. With probability at least $1 - ke^{-\alpha} = 1-p_3$, by Corollary~\ref{cor:fullrank} the submatrix $\begin{bmatrix} Y_1(J_b,:) & Y_3(J_b,:)\end{bmatrix}$ has full rank $k_1+k_3$. With probability at least  $1 - k_2 \left ( 1 - \left ( 1 - \frac{\beta}{m}\right )^{\ell_b }\right ) = 1 - p_4$, the matrix $Y_2(J_b,:)$ is made of zeros. Therefore, with probability at least $1-p_3-p_4$, the matrix $Y(J,:)$ has rank $k$. This implies that the matrix
		\begin{equation*}
			A(:,J) = X Z Y(J,:)^T
		\end{equation*}
		also has full rank, so it spans the range of the columns of $A$, and this concludes the proof.
	\end{proof}
	
	\subsection{Analysis for numerically low-rank matrices}\label{sec:approx}
	
	We now consider the case in which $A$ is a perturbation of order $\varepsilon$ of a rank-$k$ matrix of the form described in Assumption~\ref{ass:exact}. 
	
	\begin{assumption}\label{ass:approx}
		Assume that $A = XZY^T + E$ for some $X\in \mathbb{R}^{n \times k}$, $Z \in \mathbb{R}^{k \times k}$, $Y \in \mathbb{R}^{m \times k}$, and $E \in \mathbb{R}^{n \times m}$. Assume that $\mathrm{rank}(XZY^T) = k$ and $\|E\| \le \varepsilon$. Let us assume that $Z = \mathrm{blkdiag}(Z_1, Z_2, Z_3)$, and that we can partition $X$ and $Y$ as
		\begin{equation*}
			X = \begin{bmatrix} X_1 & X_2 & X_3 \end{bmatrix}, \qquad Y = \begin{bmatrix} Y_1 & Y_2 & Y_3 \end{bmatrix},
		\end{equation*}
		with $X_1\in \mathbb{R}^{n \times k_1}$, $Y_1\in \mathbb{R}^{m \times k_1}$, $X_2\in \mathbb{R}^{n \times k_2}$, $Y_2\in \mathbb{R}^{m \times k_2}$,  $X_3\in \mathbb{R}^{n \times k_3}$, $Y_3\in \mathbb{R}^{m \times k_3}$, $Z_1 \in \mathbb{R}^{k_1\times k_1}$, $Z_2 \in \mathbb{R}^{k_2\times k_2}$, $Z_3 \in \mathbb{R}^{k_3\times k_3}$, for some $k_1+k_2+k_3=k$. Let us further assume that the block matrices $\begin{bmatrix} X_1 & X_2\end{bmatrix}$ and $\begin{bmatrix} Y_1 & Y_3\end{bmatrix}$ have orthonormal columns and are $\mu$-coherent, and that $X_3$ and $Y_2$ have orthonormal columns and each column is sparse and has at most $\beta$ nonzero elements. 
	\end{assumption}
	
	The first result in this section, Theorem~\ref{thm:approx}, provides an error bound for the column subset selection error given by the index set $J$ selected in lines 1--3 of Algorithm~\ref{alg:variation}. %Later, we will prove two results on the low-rank approximation error for the CUR decomposition obtained by Algorithm~\ref{alg:basic} and a variation of that -- Algorithm~\ref{alg:variation}.
	
	\begin{theorem}\label{thm:approx}
		If Assumption~\ref{ass:approx} holds and we take
		\begin{equation*}
			\ell_0 = \alpha\mu (k_1+k_2), \quad \ell_a = k_1 + k_2, \quad \ell_b = \alpha\mu (k_1+k_3)
		\end{equation*}
		then Algorithm~\ref{alg:variation} returns a column index $J$ such that 
		\begin{equation}\label{eq:finalcolbound}
			\| A - A(:,J) A(:,J)^\dagger A \| \le \varepsilon\left (1 +  \sqrt{2} \delta_A ^{-1}\delta_C^{-1} \left ( 1 +\delta_A^2 + \delta_C^2\right )^{1/2}\right ),
		\end{equation}
		where
		\begin{equation*}
			\delta_A^{-1} = \sqrt{\frac{1+\delta}{1-\delta}} \cdot \frac{\sigma_1(XZY^T) \sqrt{km} \|Y^\dagger\|}{\sigma_k(XZY^T) - 2\varepsilon \sqrt{\frac{nmk}{(1-\delta)\ell}}}\text{ and }\delta_C^{-1} = \sqrt{\frac{m}{(1-\delta)\ell}},
		\end{equation*}
		and $\ell = \max\left \{ \ell_0, \ell_a, \ell_b\right \}$,
		with a probability of at least
		\begin{equation}\label{eq:success1}
			p = 1 - k\left ( \frac{\beta \ell}{n} + \frac{\beta \ell }{m} + \left ( \frac{e^{-\delta}}{(1-\delta)^{1-\delta}}\right )^\alpha + \left ( \frac{e^{\delta}}{(1+\delta)^{1+\delta}} \right )^{\alpha}\right ) \text{ for all } \delta \in (0,1).
		\end{equation}
	\end{theorem}
	
	The general idea of the proof of Theorem~\ref{thm:approx} is similar to the proof of Theorem~\ref{thm:exact}, but in this case, instead of showing that the selected submatrices have the right rank, we also need to prove that they are not too ill-conditioned. 
	
	\begin{proof} We divide the proof into several steps.
		\paragraph{Step \#1: A lower bound for $\sigma_{k_1 + k_2}(A(I_0,:))$.}
		Let us consider line 1 in Algorithm~\ref{alg:variation}, that is, the selection of $\ell_0$ rows uniformly at random, corresponding to the index set $I_0$. 
		We have 
		\begin{equation*}
			A(I_0,:) = X_1(I_0,:) Z_1 Y_1^T + X_2(I_0,:) Z_2 Y_2^T + X_3(I_0,:) Z_3 Y_3^T + E(I_0,:).
		\end{equation*}
		With probability at least $1 - k_3\left ( 1 - \left ( 1 - \frac{\beta}{n}\right )^{\ell_0}\right ) = 1-p_1$, the submatrix  $X_3(I_0,:)$ is made of zeros; the argument is the same as in the proof of Theorem~\ref{thm:exact}. By Theorem~\ref{thm:incoherent}, with probability at least $1 - (k_1+k_2)\left ( \frac{e^{-\delta}}{(1-\delta)^{1-\delta}}\right )^{\alpha} = 1 - p_2$, we have $$\left \lVert \begin{bmatrix} X_1(I_0,:) & X_2(I_0,:) \end{bmatrix}^\dagger \right \rVert \le \sqrt{\frac{n}{(1-\delta)\ell_0}}.$$
		Therefore, we can conclude that 
		\begin{align}\label{eq:step1}
			\sigma_{k_1 + k_2}(A(I_0,:)) & \ge \sigma_{k_1 + k_2}(X_1(I_0,:) Z_1 Y_1^T + X_2(I_0,:) Z_2 Y_2^T) - \|E(I_0,:)\|\nonumber \\
			& \ge \sigma_{k_1+k_2}\left ( X_1 Z_1 Y_1^T + X_2 Z_2 Y_2^T \right ) - \varepsilon \nonumber\\
			& \ge \sigma_k(XZY^T)\sqrt{\frac{(1-\delta)\ell_0}{n}} - \varepsilon,
		\end{align}
		where the first inequality follows from Theorem~\ref{thm:singval}(b) with $i = k_1+k_2$, $j = 1$, $F = A(I_0,:)$, and $G = -E(I_0,:)$, and the last inequality is obtained from Theorem~\ref{thm:singval}(b) by setting $i=k_1 + k_2$, $j= k_3+1$, $F = X_1 Z_1 Y_1^T + X_2 Z_2 Y_2^T$ and $G = X_3 Z_3 Y_3^T$.
		
		\paragraph{Step \#2: The matrix $A(I_0,J_a)$ is well-conditioned.}
		Let us now consider the column index set $J = (J_a, J_b)$ where $J_a$ contains the $\ell_a$ indices from sRRQR, and $J_b$ contains the new $\ell_b$ indices selected uniformly at random. 
		Since $J_a$ has been selected using sRRQR on $A(I_0,:)$, we have that
		\begin{equation}\label{eq:step2}
			\sigma_{k_1 + k_2}(A(I_0,J_a)) \ge \frac{\sigma_{k_1 + k_2}(A(I_0,:))}{\sqrt{1 + (k_1+k_2)(m-k_1-k_2)}} \ge \frac{\sigma_{k_1+k_2}(A(I_0,:))}{\sqrt{km}}.
		\end{equation}
		
		\paragraph{Step \#3: The matrix $\begin{bmatrix} Y_1(J_a,:) & Y_2(J_a,:) \end{bmatrix}$ is well-conditioned.}
		
		We have that 
		\begin{multline*}
			A(I_0,J_a)  = \begin{bmatrix} X_1(I_0,:) & X_2(I_0,:) \end{bmatrix} \begin{bmatrix} Z_1 & \\ & Z_2 \end{bmatrix} \begin{bmatrix} Y_1(J_a,:) & Y_2(J_a,:)\end{bmatrix}^T + E(I_0,J_a)\\
			= \begin{bmatrix} X_1(I_0,:) & X_2(I_0,:) \end{bmatrix} \begin{bmatrix} Z_1 & \\ & Z_2 \end{bmatrix} \begin{bmatrix} Y_1 & Y_2 \end{bmatrix}^T \left ( \begin{bmatrix} Y_1 & Y_2 \end{bmatrix}^T \right )^\dagger\begin{bmatrix} Y_1(J_a,:) & Y_2(J_a,:)\end{bmatrix}^T + E(I_0,J_a).
		\end{multline*}
		Therefore, 
		\begin{align}\label{eq:step3}
			\sigma_{k_1+k_2} & \left ( \begin{bmatrix} Y_1(J_a,:) & Y_2(J_a,:)\end{bmatrix} \right ) \ge \frac{\sigma_{k_1+k_2}(A(I_0,J_a)-E(I_0,J_a))}{\| A(I_0,:) - E(I_0,:) \| \cdot \left \lVert \begin{bmatrix} X_1(I_0,:) & X_2(I_0,:) \end{bmatrix} \right \rVert \cdot \|Y^\dagger\|}\nonumber \\
			& \ge \sqrt{\frac{n}{(1+\delta)\ell_0}} \cdot \frac{\sigma_{k_1+k_2}(A(I_0,J_a)) - \varepsilon}{\sigma_1(X(I_0,:)ZY^T) \cdot \|Y^\dagger\|}
			\ge \sqrt{\frac{n}{(1+\delta)\ell_0}} \cdot \frac{\sigma_{k_1 + k_2}(A(I_0,:)) - \varepsilon\sqrt{km}}{\sqrt{km} \cdot  \sigma_1(XZY^T) \cdot \|Y^\dagger\|} \nonumber\\
			& \ge \sqrt{\frac{1-\delta}{1+\delta}} \cdot \frac{\sigma_k(XZY^T) - 2\varepsilon \sqrt{\frac{nmk}{(1-\delta)\ell_0}}}{\sigma_1(XZY^T) \sqrt{km} \|Y^\dagger\|},
		\end{align}
		where we used Theorem~\ref{thm:singval}(a) for the first inequality, the second inequality holds with probability at least $1 - (k_1+k_3) \left ( \frac{e^{\delta}}{(1+\delta)^{1+\delta}} \right )^{\alpha} = 1-p_3$ by Theorem~\ref{thm:incoherent}, because $I_0$ was selected uniformly at random, the third inequality follows from~\eqref{eq:step2}, and the fourth from~\eqref{eq:step1}. 
		
		\paragraph{Step \#4: The matrix $Y(J,:)$ is well-conditioned.}
		
		We can partition
		\begin{equation*}
			Y(J,:) = \begin{bmatrix} Y_1(J_a, :) & Y_2(J_a, :) & Y_3(J_a, :) \\
				Y_1(J_b, :) & Y_2(J_b, :) & Y_3(J_b, :) \end{bmatrix} = \begin{bmatrix} Y_1(J_a, :) & Y_2(J_a, :) & Y_3(J_a, :) \\
				Y_1(J_b, :) & 0 & Y_3(J_b, :) \end{bmatrix},
		\end{equation*}
		where the second equality holds with a probability at least 
		\begin{equation*}
			1 - p_4 = 1 - k_2\left ( 1- \left ( 1 - \frac{\beta}{m}\right )^{\ell_b }\right ),
		\end{equation*}
		for the same reason as~\eqref{eq:prob} in the proof of Theorem~\ref{thm:exact}.
		We wish to apply Lemma~\ref{lemma:sigmamintriangular} with $A = Y_2(J_a,:)$, $B = \begin{bmatrix} Y_1(J_a,:) & Y_3(J_a,:)\end{bmatrix}$ and $C = \begin{bmatrix} Y_1(J_b,:) & Y_3(J_b,:)\end{bmatrix}$. Using Theorem~\ref{thm:incoherent} again, we have that
		\begin{equation*}
			\left \lVert \begin{bmatrix} Y_1(J_b,:) & Y_3(J_b,:) \end{bmatrix}^\dagger \right \rVert \le \sqrt{\frac{m}{(1-\delta)\ell_b}}
		\end{equation*}
		with probability at least $1-(k_1+k_3) \left ( \frac{e^{-\delta}}{(1-\delta)^{1-\delta}} \right )^\alpha = 1 - p_4$. We have that $$\left \lVert \begin{bmatrix} Y_1(J_a,:) & Y_3(J_a,:) \end{bmatrix} \right \rVert \le 1$$ because this is a submatrix of a matrix with orthonormal columns. Furthermore, we have that
		\begin{equation*}
			\sigma_{k_2}(Y_2(J_a,:)) \ge \sigma_{k_1+k_2}\left ( \begin{bmatrix} Y_1(J_a,:) & Y_2(J_a,:) \end{bmatrix} \right )
		\end{equation*}
		and the right-hand-side can be bounded by~\eqref{eq:step3}. Therefore, the assumptions of Lemma~\ref{lemma:sigmamintriangular} are satisfied with 
		\begin{equation*}
			\delta_A^{-1} = \sqrt{\frac{1+\delta}{1-\delta}} \cdot \frac{\sigma_1(XZY^T) \sqrt{km} \|Y^\dagger\|}{\sigma_k(XZY^T) - 2\varepsilon \sqrt{\frac{nmk}{(1-\delta)\ell_0}}}, \quad \delta_C^{-1} = \sqrt{\frac{m}{(1-\delta)\ell_b}}, \quad b = 1,
		\end{equation*}
		so
		\begin{equation}\label{eq:YJ+}
			\|Y(J_1,:)^\dagger\| \le \delta_A ^{-1}\delta_C^{-1} \left ( 1 + \delta_A^2 + \delta_C^2\right )^{1/2}.
		\end{equation}
		
		\paragraph{Step \#5: The columns are a good subset for low-rank approximation.}
		
		In order to bound $\|A - A(:,J)A(:,J)^\dagger A\|$ we apply Theorem~\ref{thm:chiu} to the matrix $A$, the unitary matrix $U = \begin{bmatrix} U_1 & U_2 \end{bmatrix}$ where $U_1$ is an orthonormal basis of $Y$ (which has rank $k$) and $U_2$ is an orthogonal complement of $Y$, and the index set $J$. We have that $\|A U_2 \| = \|E U_2 \| \le \varepsilon$, and $\|U_2(J,:)\| \le 1$. To bound $\|U_1(J,:)^\dagger\|$ let us consider the skinny QR decomposition $Y = U_1 R$ for some invertible matrix $R \in \mathbb{R}^{k \times k}$; then
		\begin{equation*}
			\|U_1(J,:)^\dagger\| = \|R \cdot Y(J,:)^\dagger\| \le \|Y(J,:)^\dagger \| \cdot \|Y\| \le \sqrt{2}\delta_A ^{-1}\delta_C^{-1} \left ( 1 + \delta_A^2 + \delta_C^2\right )^{1/2},
		\end{equation*}
		where the last equality follows from~\eqref{eq:YJ+} and the fact that $Y$ has two blocks of orthonormal columns. Applying Theorem~\ref{thm:chiu} then gives
		\begin{equation}\label{eq:colbound}
			\|A - A(:,J) A(:,J)^\dagger A\| \le \varepsilon (1 +  \sqrt{2}\delta_A ^{-1}\delta_C^{-1} \left ( 1 + \delta_A^2 + \delta_C^2\right )^{1/2}). 
		\end{equation}
		
		\paragraph{Step \#6: The failure probability.}
		
		The success probability $p$ can be bounded using a union bound on all the probabilistic results used in this proof: we have 
		% \begin{scriptsize}
			\begin{align*}
				p & \ge 1 - (p_1 + p_2 + p_3 + p_4) \\
				& \ge 1 - k_3\left ( 1 - \left ( 1 - \frac{\beta}{n}\right )^{\ell_0} \right )  - (k_1+k_2)\left ( \frac{e^{-\delta}}{(1-\delta)^{1-\delta}}\right )^{\alpha} \\
				& \qquad - (k_1+k_3) \left ( \frac{e^{\delta}}{(1+\delta)^{1+\delta}} \right )^{\alpha} - k_2\left ( 1 - \left ( 1 - \frac{\beta}{m}\right )^{\ell_b }\right )\\
				& \ge \left ( 1-\frac{\beta \ell_0 k_3}{n} \right ) + \left ( 1 - \frac{\beta \ell_b k_2}{m} \right ) - k\left ( \frac{e^{-\delta}}{(1-\delta)^{1-\delta}}\right )^\alpha - k \left ( \frac{e^{\delta}}{(1+\delta)^{1+\delta}} \right )^{\alpha},
			\end{align*}
			% \end{scriptsize}
		which concludes the proof of the theorem.
	\end{proof}
	
	We provide a simplified bound below that ``hides'' all the constants of Theorem~\ref{thm:approx} that distract from the essence of the result.
	\begin{corollary}\label{cor:bigO}
		In the same assumptions as Theorem~\ref{thm:approx}, we have that
		\begin{equation*}
			\|A - A(:,J)A(:,J)^\dagger A\| \le \mathcal O\left (\varepsilon \sqrt{\frac{nmk}{\ell}} \cdot \frac{\sigma_1(XZY^T)}{\sigma_k(XZY^T)} \cdot \|Y^\dagger\|\right )
		\end{equation*}
		with a probability of at least
		\begin{equation*}
			1 - k\left ( \frac{\beta \ell}{n} + \frac{\beta \ell}{m} + c_1^{\alpha} + c_2^{\alpha}\right ).
		\end{equation*}
		for some fixed constants $c_1, c_2 \in (0,1)$.
	\end{corollary}
	For example, if we fix $\delta = 4/5$, the success probability~\eqref{eq:success1} gives $c_1 \approx 0.62$ and $c_2 \approx 0.78$. For this corollary (and Theorem~\ref{thm:approx}) to make sense, we need to have $\varepsilon \sqrt{nmk/\ell} \ll \sigma_k(XZY^T)$, that is, the magnitude of the perturbation $E$ needs to be small compared to the smallest ``relevant'' singular value of $A$.
	
	As mentioned in the introduction, there is a well-known way to relate the row/column subset selection error to the low-rank approximation error; see, e.g.,~\cite{Cortinovis2020,Drineas2008,Saibaba2016,Sorensen2016}. In particular, for any index sets $I$ and $J$, we have
	\begin{align}
		\|A - & A(:,J)A(:,J)^\dagger A A(I,:)^\dagger A(I,:) \|_2 \nonumber \\
		& \le \|A - A(:,J)A(:,J)^\dagger A \| + \| A(:,J)A(:,J)^\dagger (A - A A(I,:)^\dagger A(I,:))\|\nonumber \\
		& \le \|A - A(:,J)A(:,J)^\dagger A \| + \| A(:,J)A(:,J)^\dagger \| \cdot \|A - A A(I,:)^\dagger A(I,:)\| \nonumber \\
		& \le \|A - A(:,J)A(:,J)^\dagger A \| + \|A - A A(I,:)^\dagger A(I,:)\|, \label{eq:orthogonal}
	\end{align}
	where the third inequality follows from the fact that $A(:,J)A(:,J)^\dagger$ is an orthogonal projection. 
	
	\begin{corollary}\label{cor:variation}
		Given a matrix $A$ satisfying the assumptions of Theorem~\ref{thm:approx} and given the output index sets $I$ and $J$ from Algorithm~\ref{alg:variation} with $\ell_0 = \ell_b = \alpha \mu k$ and $\ell_a = k$, we have that
		\begin{equation*}
			\left \lVert A - A(:,J)A(:,J)^\dagger A A(I,:)^\dagger A(I,:) \right \rVert \le \mathcal O\left (\varepsilon \sqrt{\frac{nmk}{\ell}} \cdot \frac{\sigma_1(XZY^T)}{\sigma_k(XZY^T)} \cdot (\|X^\dagger\| + \|Y^\dagger\|)\right )
		\end{equation*}
		with a probability of at least 
		\begin{equation*}
			1 - 2k\left ( \frac{\beta \ell}{n} + \frac{\beta \ell}{m} + c_1^{\alpha} + c_2^{\alpha}\right ).
		\end{equation*}
	\end{corollary}
	\begin{proof}
		The bound follows from the inequality~\eqref{eq:orthogonal} and Theorem~\ref{thm:approx} applied to $A$ and $A^T$. The bound on the failure probability follows from a union bound of the failure probabilities given by Theorem~\ref{thm:approx} on $A$ and $A^T$.
	\end{proof}
	
	% \paragraph{Random error.} It makes sense to think that, if $A$ is approximately a rank-$k$ matrix, the error $E$ may be random. This, however, does not help our analysis too much. Let us consider, for example, a matrix $A = XZY^T + E$ where $X, Y, Z$ satisfy the assumptions of Theorem~\ref{thm:approx} and $E \in \mathbb{R}^{n \times m}$ is made of independent identically distributed (i.i.d.) entries $e_{ij} \sim N\left ( 0, \frac{\varepsilon}{\sqrt{2nm}} \right )$. The variance is chosen in such a way that, for large values of $n,m$, with high probability $\|E\| \approx \varepsilon$, which is due to Wigner's semi-circle law; see, e.g.,~\cite[Section 1.9]{Mingo2017} for a reference. Under this assumption, the bound from Step \#1 of the proof can be slightly improved. Indeed, with high probability we have $$\|E(I_0,:)\| \approx \frac{\varepsilon}{2} \left (1 + \sqrt{\frac{\ell}{n}}\right ),$$
	% which improves the dependence with respect to $\varepsilon$ of the denominator of $\delta_A$ by a factor $2$. 
	
	\section{Iterative randomized sRRQR algorithm}\label{sec:iterative}
	
	In this section, we discuss a variation of Algorithm~\ref{alg:variation} that is inspired from~\cite[Algorithm 2.2]{Li2015}. We first present the algorithm (Algorithm~\ref{alg:basic}) and then detail its relation to Algorithm~\ref{alg:variation} and the literature.
	
	We start from $\ell$ randomly selected rows of $A$, corresponding to an index set $I_0$. We perform the sRRQR algorithm on the submatrix $A(I_0,:) \in \mathbb{R}^{\ell \times m}$ with parameter $\ell/2$, and select a column index set $J_1$ made of the corresponding columns selected by sRRQR and another $\ell/2$ columns selected uniformly at random. Now we look at submatrix $A(:,J_1) \in \mathbb{R}^{n \times \ell}$ and perform the sRRQR algorithm on its transpose, with parameter $\ell/2$. This gives us a new set of row indices, to which we add $\ell/2$ new rows, sampled uniformly at random, and merge them into the new row index set $I_1$. This procedure can be iterated a couple of times by alternatively optimizing the columns and the rows. The cost of each iteration is $\mathcal O\left (mn\ell \log_{\eta}(\min\{n,m\})\right )$; in particular, for small values of $\ell$ this is sublinear in the size of $A$. The value of $\ell$ need not be the same in each step of the algorithm; see Algorithm~\ref{alg:basic} below.  The number of rows and columns should be chosen to be slightly larger than $k$, the target rank for the low-rank approximation.
	
	\begin{algorithm}
		\caption{Iterative randomized sRRQR for row and column selection}
		\label{alg:basic}
		\begin{algorithmic}[1]
			\REQUIRE{Matrix $A \in \mathbb{R}^{n \times m}$, iterations $H$, numbers $\ell_0, \ell_h^{\text{sRRQR, col}}, \ell_h^{\text{new, col}},\ell_h^{\text{sRRQR, row}}, \ell_h^{\text{new, row}}$}
			\STATE{Select $\ell_0$ row indices $I_0$ uniformly at random}
			\FOR{$h = 1, \ldots, H$}
			\STATE{Apply sRRQR algorithm to $A(I_{h-1},:)$ with parameter $\ell_h^{\text{sRRQR, col}}$}
			\STATE{Build column index set $J_h$ from the $\ell_h^{\text{sRRQR, col}}$ indices given by sRRQR factorization and $\ell_h^{\text{new, col}}$ new indices selected uniformly at random}
			\STATE{Apply sRRQR algorithm to $A(:,J_h)^T$ with parameter $\ell_h^{\text{sRRQR, row}}$}
			\STATE{Build row index set $I_h$ from the $\ell_h^{\text{sRRQR, row}}$ indices given by sRRQR factorization and $\ell_h^{\text{new, row}}$ new indices selected uniformly at random}
			\ENDFOR
		\end{algorithmic}
	\end{algorithm}
	
	\begin{remark}
		In Algorithm~\ref{alg:basic}, some of the row/column indices of $A$ that have been sampled get thrown away at each step; this has been done because it simplifies our analysis in the rest of the paper. In practice, one can use the extra information contained in those rows and columns and consider final index sets $I = I_0 \cup I_1 \cup \cdots \cup I_{H}$ and $J = J_0 \cup J_1 \cup \cdots \cup J_{H}$. The corresponding low-rank approximation will be slightly more expensive to compute but will also yield a smaller error. Indeed, if $I \subseteq I'$ and $J \subseteq J'$ then
		\begin{equation*}
			\|A - A(:,J')A(:,J')^\dagger A A(I',:)^\dagger A(I',:)\| \le \|A - A(:,J)A(:,J)^\dagger A A(I,:)^\dagger A(I,:)\|.
		\end{equation*}
	\end{remark}
	
	Some remarks are in order:
	\begin{itemize}
		\item Algorithm~\ref{alg:basic} is almost the same as Algorithm 2.2 in \cite{Li2015}, with the difference that we use a sRRQR factorization instead of a pivoted QR factorization, which allows us to obtain theoretical guarantees on the quality of the selected rows and columns.
		\item A strategy that is similar to Algorithm~\ref{alg:basic} has been proposed in~\cite{Goreinov2010} in the context of algorithms for cross approximation; in~\cite{Goreinov2010}, the authors employ a volume maximization technique, closely related to sRRQR factorizations (see~\cite{Gu1996}).
		\item The first three steps of Algorithm~\ref{alg:basic} (lines 1, 3, 4) coincide with the first three lines of Algorithm~\ref{alg:variation}; what changes here afterward is that, instead of restarting the sampling process from scratch to select the row indices, one uses the selected columns at a starting point. The advantage of Algorithm~\ref{alg:basic} is that the process can be iterated, and the quality of the low-rank approximation improves in many cases, in practice (see Example~\ref{ex:hilbert}). 
	\end{itemize}

	\subsection{Numerical examples}\label{sec:examples}
	Let us discuss the behavior of Algorithm~\ref{alg:basic} on some examples. The numerical experiments have been performed in Matlab, and the implementation of the sRRQR algorithm from~\cite{Gu1996} is taken from~\cite{sRRQR}, with parameter $\eta = 1.1$. % \LY{maybe not including the link in the main text, instead treating it as a reference? - Fixed!}
	\begin{example}\label{ex:notwork2}
		Let us consider again the matrix $A$ from Example~\ref{ex:notwork}. Also, for  Algorithm~\ref{alg:basic}, it is very unlikely to end up selecting the first row and column. Indeed, the initial selection of $\ell_0$ rows will likely not include the first row, and since the first column will be zero in the selected submatrix, the first column will not be included in the column index set. 
	\end{example}
	
	\begin{example}\label{ex:rank2-2}
		Let us consider again the matrix from Example~\ref{ex:rank2} and let us show that Algorithm~\ref{alg:basic} is able to recover $A$ successfully. The row set $I_0$, chosen uniformly at random, will most likely not include the first row. However, when looking at the matrix $A(I_0,:)$, the strong RRQR algorithm will select a set $J_1$ containing the first column, plus some other columns sampled uniformly at random. Now, when applying strong RRQR on the column subset $A(:,J_1)$, we will get an index set $I_1$ that includes the first row. At this point, the index sets $I_1$ and $J_1$ are enough to recover the rank-2 matrix $A$. 
	\end{example}
	
	\begin{example}\label{ex:function}
		Let us consider the bivariate function
		\begin{equation*}
			f(x,y) = \frac{5\sin(3x)}{5y-4} + 2e^{x/2}\cos(10y) + \frac{20y}{4x-1}.
		\end{equation*}
		We construct the matrix $A$ as the discretization of $f(x,y)$ on an equispaced grid of size $1000 \times 1000$ in the square $[0,1]^2$, to which we add a random matrix $E$ with norm $\|E\| = 10^{-5}$. The matrix $A$ is of rank approximately $3$, and some of the first three singular vectors have high coherence. A visual representation of $A$ and of its first three right and left singular vectors is in Figure~\ref{fig:matrix}. We run Algorithm~\ref{alg:basic} with $\ell_0 = 6$ and all other $\ell$-parameters set to $3$. After one full iteration, we are able to find a good subset of rows and columns. The quality does not improve with more iterations; see Figure~\ref{fig:plot} (left), where we also compare this with the low-rank approximation error given by the index sets returned by Algorithm~\ref{alg:variation} with $\ell_0 = 6, \ell_a = \ell_b = 3$.
		\begin{figure}[htb]
			\centering
			\begin{subfigure}[b]{0.48\textwidth}
				\centering
				\includegraphics[scale=.5]{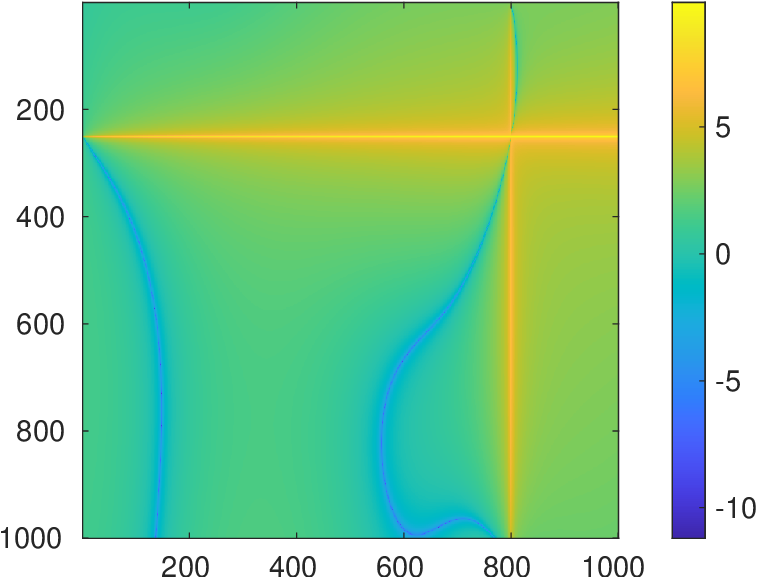}
			\end{subfigure}
			\begin{subfigure}[b]{0.48\textwidth}
				\centering
				\includegraphics[scale=.5]{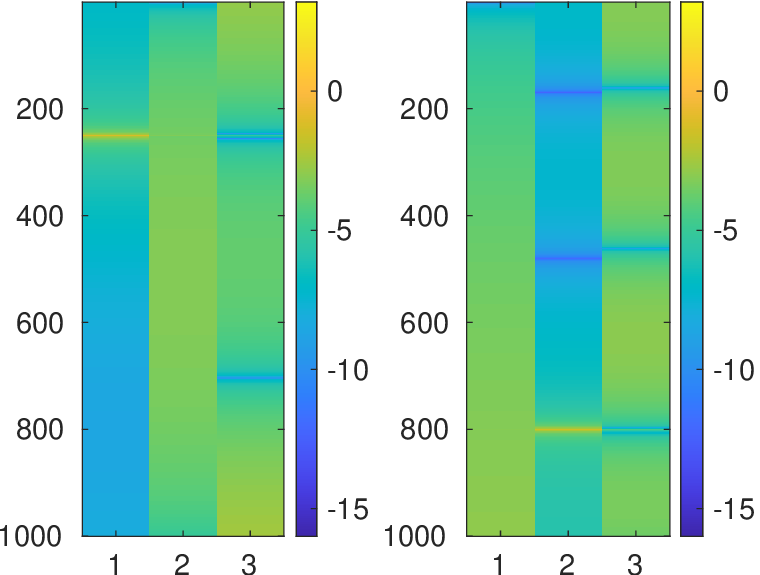}
			\end{subfigure}
			
			\caption{Logarithm of the magnitude of the entries of the matrix $A$ from Example~\ref{ex:function} (left) and of its first three left and right singular vectors (middle and right). The first left singular vector and the second right singular vectors are approximately sparse, while all the other ones are incoherent.}
			\label{fig:matrix}
		\end{figure}
	\end{example}
	
	\begin{example}\label{ex:hilbert}
		While our analysis in Section~\ref{sec:erroranalysis2} is restricted to the first full iteration of Algorithm~\ref{alg:basic}, it is worth noting that, in many cases, the low-rank approximation error continues to decrease if more iterations are performed. As an example of this behavior, let us consider the matrix $A\in\mathbb{R}^{1000 \times 1000}$ with entries $a_{ij} = \frac{1}{i + j^2 + 1}$, and let us run Algorithm~\ref{alg:basic} with $\ell_0 = 6$ and all other $\ell$-parameters set to $5$. The result is shown in Figure~\ref{fig:plot} (right), where we also compare this with the low-rank approximation error given by the index sets returned by Algorithm~\ref{alg:variation} with $\ell_0 = 10, \ell_a = \ell_b = 5$.
		\begin{figure}[htb]
			% \begin{subfigure}[b]{0.48\textwidth}
				\centering
				\includegraphics[scale=.58]{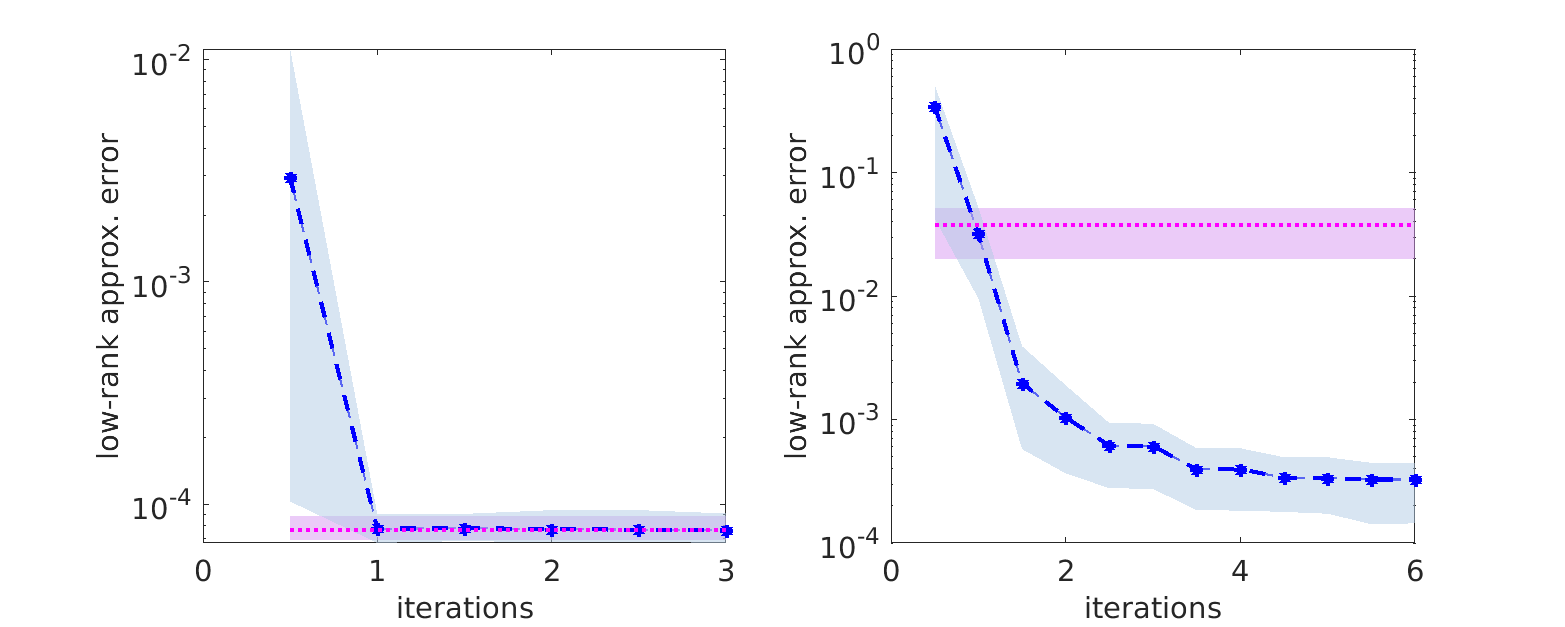}
				% \end{subfigure}
			% \begin{subfigure}[b]{0.48\textwidth}
				% \centering
				% \includegraphics[scale=.5]{hilbert.eps}
				% \end{subfigure}
			\caption{The blue part of the plot indicates the low-rank approximation error, namely, the quantity $\|A - A(:,J)A(:,J)^\dagger A A(I,:)^\dagger A(I,:)\|$, where $I$ and $J$ are the index sets returned by Algorithm~\ref{alg:basic} for the matrix from Example~\ref{ex:function} (left) and the one from Example~\ref{ex:hilbert} (right). The $x$-axis denotes the number of iterations of Algorithm~\ref{alg:basic}. We run Algorithm~\ref{alg:basic} 100 times; the blue dashed line represents the average error, and the light blue areas represent the 90\% confidence region. The dotted magenta line is the approximation error given by the index sets returned by Algorithm~\ref{alg:variation} (it is constant because the method is not iterative), and the pink area is the 90\% confidence region. Note that, while for Example~\ref{ex:function} the two algorithms give similar errors after the first step, a few iterations of Algorithm~\ref{alg:basic} help improve the error in the case of Example~\ref{ex:hilbert}.}
			\label{fig:plot}
		\end{figure}
	\end{example}

	\subsection{Error analysis}\label{sec:erroranalysis2}
	
	We restrict the analysis of Algorithm~\ref{alg:basic} to its first full step. We provide a result that states that exact reconstruction happens with high probability when the matrix $A$ is exactly low-rank and satisfies Assumption~\ref{ass:exact}, and an error bound in the case in which $A$ has been perturbed by error, similarly to what we have done for the analysis of Algorithm~\ref{alg:variation}.
	
	\begin{theorem}\label{thm:exact_variation}
		If Assumption~\ref{ass:exact} holds and we take
		\begin{equation*}
			\ell_0 = \alpha\mu (k_1+k_2), \; \ell_1^{\mathrm{sRRQR, col}} = k_1 + k_2, \; \ell_1^{\mathrm{new, col}} = \alpha\mu (k_1+k_3), \; \ell_1^{\mathrm{sRRQR, row}} = k, \; \ell_1^{\mathrm{new, row}} = 0,
		\end{equation*}
		then one iteration of Algorithm~\ref{alg:basic} returns sets of rows $I_1$ and columns $J_1$ such that 
		\begin{equation*}
			A = A(:,J_1)A(:,J_1)^\dagger A A(I_1,:)^\dagger A(I_1,:)
		\end{equation*}
		with a probability of at least
		\begin{equation*}
			1 - \left ( \frac{\beta \ell k}{n} + \frac{\beta \ell k}{m} + 2ke^{-\alpha} \right ),
		\end{equation*}
		where $\ell = \max\left \{ \ell_0, \ell_1^{\mathrm{sRRQR, col}}, \ell_1^{\mathrm{new, col}}, \ell_1^{\mathrm{sRRQR, row}}\right \} \le \alpha\mu k$.
	\end{theorem}
	
	\begin{proof}
		Since lines 1--4 in Algorithm~\ref{alg:basic} coincide with lines 1--3 in Algorithm~\ref{alg:variation}, the proof of Theorem~\ref{thm:approx} shows that with probability at least 
		\begin{align*}
			1 - p_1 -p_2-p_3-p_4 \ge 1 - \left ( \frac{\beta \ell k}{n} + \frac{\beta \ell k}{m} + 2ke^{-\alpha} \right ),
		\end{align*}
		we have that $A(:,J_1) = XZY(J_1,:)^T$ has full rank $k$. Line 5 in Algorithm~\ref{alg:basic} will then select a row-index set $I_1$ such that $A(I_1, J_1)$ is a rank-$k$ matrix. Hence, we have that $A(I_1,:)$ spans the range of the rows of $A$, so the index sets $I_1$ and $J_1$ guarantee that the matrix $A$ is reconstructed exactly when forming the CUR approximation $A(:,J_1)A(:,J_1)^\dagger A A(I_1,:)^\dagger A(I_1,:)$.
	\end{proof}
	
	Let us now look at the perturbed case. Since the first half iteration of Algorithm~\ref{alg:basic} coincides with lines 1--3 of Algorithm~\ref{alg:variation}, Theorem~\ref{thm:approx} gives a bound on the column subset selection error obtained using the index set $J_1$ obtained in the first half iteration. 
	The following result allows us to obtain an upper bound for the low-rank approximation error of Algorithm~\ref{alg:basic} after one full iteration.
	
	\begin{theorem}[{\cite[Theorem 1]{Dong2023}}]\label{thm:dong}
		Given a ``column space approximator'' $B \in \mathbb{R}^{n \times \ell}$ of $A$ that has full column rank, let $C_1 \in \mathbb{R}^{\ell \times \ell}$ be a subset of $\ell$ linearly independent rows of $B$ and let $C_2 \in \mathbb{R}^{(n-\ell) \times \ell}$ be the set formed by all other rows of $B$. Let $R$ be the rows of $A$ corresponding to $C_1$. Then
		\begin{equation*}
			\|A - A R^\dagger R\| \le \sqrt{1 + \| C_2C_1^\dagger\|_2^2} \cdot \|A - BB^\dagger A\|.
		\end{equation*}
	\end{theorem}
	
	\begin{theorem}\label{thm:1step}
		If Assumption~\ref{ass:approx} is satisfied and we set
		\begin{multline*}
			\ell_0 = \alpha\mu (k_1+k_2), \quad \ell_1^{\mathrm{sRRQR, col}} = k_1 + k_2, \quad \ell_1^{\mathrm{new, col}} = \alpha\mu (k_1+k_3), \quad  \ell_1^{\mathrm{sRRQR, row}} = k, \\
			\ell_1^{\mathrm{new, row}} = 0, \quad
			\ell_1^{\mathrm{sRRQR, row}} = \ell_1^{\mathrm{sRRQR, col}} + \ell_1^{\mathrm{new, col}},  \quad \ell_1^{\mathrm{new, row}} = 0,
		\end{multline*}
		then the index sets $I_1$ and $J_1$ selected after one full iteration of Algorithm~\ref{alg:basic} satisfy
		\begin{equation*}
			\|A - A(:,J_1)A(:,J_1)^\dagger A A(I_1,:)^\dagger A(I_1,:)\| \le O\left (\varepsilon n\sqrt{mk} \cdot \frac{\sigma_1(XZY^T)}{\sigma_k(XZY^T)} \cdot (\|X^\dagger\| + \|Y^\dagger\|)\right )
		\end{equation*}
		with a probability of at least 
		\begin{equation*}
			1 - k\left ( \frac{\beta \ell}{n} + \frac{\beta \ell}{m} + c_1^{\alpha} + c_2^{\alpha}\right ).
		\end{equation*}
	\end{theorem}
	
	\begin{proof}
		Using~\eqref{eq:orthogonal}, we can split the error into two terms; the first one is directly bounded by Theorem~\ref{thm:approx}, and we can use Theorem~\ref{thm:dong} to bound the second term $\|A - A A(I_1,:)^\dagger A(I_1,:)\|$. More specifically, we apply Theorem~\ref{thm:dong} with $B = A(:,J_1)$; note that, thanks to the sRRQR properties, we have 
		\begin{equation*}
			\|  A(I_1^c,J_1) A(I_1,J_1)^{-1} \| \le \sqrt{\ell (n-\ell)} \|  A(I_1^c,J_1) A(I_1,J_1)^{-1} \|_{\max} \le \sqrt{\ell (n-\ell)},
		\end{equation*}
		where we denoted by $I_1^c \in \{1, \ldots, n\}^{n - |I_1|}$ the index set corresponding to all the rows not selected by $I_1$. It follows that
		\begin{equation*}
			\|A - A A(I_1,:)^\dagger A(I_1,:)\| \le \sqrt{1 + \ell(n-\ell)} \|A - A(:,J_1)A(:,J_1)^\dagger A\|
		\end{equation*}
		and the quantity $\|A - A(:,J_1)A(:,J_1)^\dagger A\|$ can be bounded, again, using Theorem~\ref{thm:approx}.
	\end{proof}

	\section{Discussion}\label{sec:conclusion}
	
	In this work, we have shown that, with high probability, a uniform sampling strategy combined with sRRQR (Algorithm~\ref{alg:variation}) selects rows and columns that give a good CUR approximation for matrices satisfying Assumption~\ref{ass:approx}. Our error bounds are informative when the perturbation $E$ to the low-rank matrix $XZY^T$ is small enough that $\|E\|  \ll \sqrt{nmk/\ell}$ and the ratio $\sigma_1(XZY^T) / \sigma_k(XZY^T)$ is not too large. Some dependence of the error on the size of the matrix is inevitable in the spectral norm when working with row and column-based low-rank approximations; see~\eqref{eq:bestspectral} and the results in~\cite{Chiu2013}. %The non-tightness of our main result -- Theorem~\ref{thm:approx} -- is because we have used worst-case scenarios for singular value inequalities, for instance, when bounding the $(k_1 + k_2)$th singular value of $A(I_0,:)$ in Step \#1 of the proof. The bounds on the failure probability are also pessimistic; this is due to the use of the union bound and Theorem~\ref{thm:incoherent}. 
	While Algorithm~\ref{alg:variation} is reminiscent of~\cite[Algorithm 3]{Chiu2013}, a key difference is that Algorithm~\ref{alg:variation} requires some additional random rows and columns after the sRRQR step (lines 3 and 6), which are necessary to obtain good low-rank approximation for a larger class of matrices than the one considered in~\cite{Chiu2013}, that is, matrices that admit a decomposition into factors with suitable incoherence and/or sparsity assumptions up to a small additive error. 
	
	While it is difficult, in general, to check whether a matrix $A$ admits a decomposition satisfying Assumption~\ref{ass:approx}, the objective of this paper is to shed some light on the excellent practical performance of simple sublinear-time algorithms for column and row subset selection. %, for a larger subset of matrices than the ones considered in~\cite{Chiu2013}. 
	Note that, in Assumption~\ref{ass:approx}, it is easier to think of $XZY^T$ as the singular value decomposition of $A$ or its best rank-$k$ approximation, but in fact, we do not require $X$ and $Y$ to have orthonormal columns, as long as they are well-conditioned. This flexibility allows us to apply our bounds to a larger class of matrices. 
	
	In Section~\ref{sec:iterative}, we introduced Algorithm~\ref{alg:basic}, directly inspired by~\cite[Algorithm 2.2]{Li2015}, and we provided an error bound for the CUR approximation constructed with the indices returned after one iteration. Note that, with respect to Theorem~\ref{thm:approx} on the column subset selection error, the constant
	in the bound of Theorem~\ref{thm:1step} has been multiplied approximately by $\sqrt{n\ell}$. This increase in the error is a result of Theorem~\ref{thm:dong}, and we have not seen it in practice in the numerical examples reported in Section~\ref{sec:examples}.

	Our results do not cover all the matrices for which \emph{there is hope}, in the sense of Section~\ref{sec:hope}. For example, a scenario that is not covered by the current theory and is left for future work consists of matrices that have some incoherence in the blocks $X_1, X_2, Y_1, Y_3$ but do not satisfy any sparseness assumption. Another interesting direction for future research is the analysis of Algorithm~\ref{alg:basic} beyond the first iteration.
	
	% \appendix
	
	% \section{A linear algebra results for block triangular matrices}

	\bibliographystyle{abbrv}
	\bibliography{Bib} 
	
\end{document}